\documentclass[twoside,final,smallextended]{svjour3}
\usepackage[latin9]{inputenc}
\usepackage{url}
\usepackage{amsmath}
\usepackage{amssymb}

\usepackage{epsfig}
\makeatletter

\numberwithin{equation}{section}

\usepackage{xcolor}\usepackage{graphicx}\usepackage{xspace}\usepackage{colortbl}\usepackage{rotating}%
\usepackage[raggedrightboxes]{ragged2e}
\usepackage{amsfonts}
\graphicspath{{X3017abArxivePrtoJrn_graphics/}{X3017abArxivePrtoJrn_tcache/}{X3017abArxivePrtoJrn_gcache/}}
\DeclareGraphicsExtensions{.pdf,.eps,.ps,.png,.jpg,.jpeg} 
\newtheorem{condition}[proposition]{Condition}
\input{tcilatex}

\makeatother

\numberwithin{equation}{section}

\input{tcilatex}

\AtBeginDocument{
  
}

\makeatother

\begin{document}
\title{EA-Matrix integrals of associative algebras and equivariant localization\protect\thanks{ Preprint CNRS HAL-00507788 (2010) }. }

\author{Serguei Barannikov}
\date{}

\institute{IMJ-UMR7586 CNRS,  l'Universit\'e de Paris; NRU-HSE \\ \email{serguei.barannikov@imj-prg.fr}}

\dedication{To R.K.Gordin on the occasion of his 70th birthday} 
\maketitle
\begin{abstract}
The EA-matrix integrals, introduced in \cite{B06},
are studied in the case of graded associative algebras with odd or
even scalar product. I prove that the EA-matrix integrals for associative algebras
with scalar product are integrals of equivariantly closed differential forms with respect to Lie algebra $gl_{N}(A)$.
\end{abstract}
\section*{Introduction}

The theory of periods of noncommutative varieties, depending on commutative parameters,  was introduced in \cite{B00}. The analogue of top-degree holomorphic form in this setting was shown in \textit{loc.cit.} to be certain element of semi-infinite subspace of negative cyclic homology. The integrals of this element satisfy the second order equation with respect to the parameters of deformations of the varieties. It was proven in \textit{loc.cit.} that the generating function of genus zero Gromov-Witten invariants of complete intersection in $\mathbb{CP}^d$  with trivial canonical class coincides with the coefficient of this  second order equation for the family of mirror varieties. This approach had singled out the $A_{\infty}-$ algebras/categories, satisfying cyclic homology analogue of degeneration of Hodge to de Rham spectral sequence, as the proper definition of (smooth and compact) noncommutative varieties. 

The EA-matrix integrals were introduced in  \cite{B06} as a set of
periods of associative, more generally $A_{\infty}-$algebras or
noncommutative varieties, depending on \textit{noncommutative} parameters:
\[
\mathcal{F}(Y)=\int_{\Gamma}\exp(Tr\langle Y,X\rangle+\frac{1}{3!}m_{\tilde{A}}(X,\frac{\partial}{\partial  X}))\vdash\prod_{\alpha,i,j}dX_{i}^{\alpha,j}
\]
 $\tilde{A}=A\otimes q_{N} $/  $\tilde{A}=A\otimes gl_{N} $  in even/odd  scalar product case, here $q_N$ is the odd matrix algebra, see \textit{loc.cit.}
It was shown in theorem 3 in \textit{loc.cit.} that
the matrix Airy integral from \cite{K}   corresponds in this way to the simplest associative
algebra of one dimension $A=\{e|e^2=e\}$. 

The usual varieties correspond here to $A_{\infty}-$algebras of endomorphisms of generators of their $D^b(\textrm{Coh})-$categories.

The asymptotic expansion of EA-matrix integrals via BV formalism
was shown in  \cite{B06,B18a} to define,  as a sum over generalized ribbon
graphs, a generating function for series of cohomology classes of compactified moduli spaces
of curves of all genus. A particular example is the formula  
for cohomology-valued generating function for products of $\psi$-classes,
$\psi_{i}=c_{1}(T_{p_{i}}^{*})$, in the cohomology $H^{*}(\bar{\mathcal{M}}_{g,n})$ calculated by the stable ribbon graph complex (\cite{B18a}):
\begin{equation}
\sum_{\sum d_{i}=d}\psi_{1}^{d_{1}}\ldots\psi_{n}^{d_{n}}\prod_{i=1}^{n}\frac{(2d_{i}-1)!!}{\lambda_{i}^{(2d_{i}+1)}}= \left[ \sum_{G\in\Gamma_{g,n}^{dec,odd}} G \,\, \frac{2^{-\chi(G)}}{\left|\textrm{Aut}(G)\right|}\prod_{e\in\textrm{Edge}(G)}\frac{1}{\lambda_{i(e)}+\lambda_{j(e)}}  \right]  \label{eq:psiprod}
\end{equation}
where the sum on the right is over \emph{stable ribbon} oriented graphs of
genus $g$ with $n$ numbered punctures, with $2d+n$ edges, and such
that at each vertex the cyclically ordered subsets of edges have arbitrary \emph{odd}
cardinality. 

In this paper it is proven that the EA-matrix integrals for associative algebras
with scalar product are integrals of equivariantly closed differential forms
with respect to the Lie algebra $gl_{N}(A)$. This generalizes and clarifies the similar result with respect to the Lie algebra $gl_{N}$ from \cite{B10}. The localization
formula for the $gl_{N}(A)-$action then leads in \cite{B18} to calculation of
these EA matrix integrals via determinants and $\tau-$functions of
integrable hierarchies. 

The paper is dedicated to Raphail Kalmanovich Gordin, with gratitude.

\subsection*{Acknowledgements}
The results of this paper were presented starting from 2010 at conferences in Miami, Vienna, Tokyo, Moscow, Bonn, Boston. I am thankful to organizers of these conferences for their hospitality and for the opportunity to present these results to large audience, and to participants for interesting questions.

\subsection*{Notations}
For a $\mathbb{Z}/2\mathbb{Z}$-graded vector space $A=A_{0}\oplus A_{1}$
denote via $\Pi A$ the parity inverted vector space, $\left(\Pi A\right)_{0}=A_{1}$,
$\left(\Pi A\right)_{1}=A_{0}$. For an element $a$ from $\mathbb{Z}/2\mathbb{Z}$-graded
vector space $A$ denote by$\pi a$$\in\Pi A$ the same element considered
with parity reversed. 

\section{Equivariantly closed De Rham differential form.}
Let $A=A_{0}\oplus A_{1}$ denotes a $\mathbb{Z}/2\mathbb{Z}$-graded
associative algebra, $\dim_{k}A_{0}=r<\infty$, $char(k)=0$, with
multiplication denoted by $m_{2}:A^{\otimes2}\rightarrow A$. Let
$A$ be endowed with \emph{odd} invariant non-degenerate scalar product $\langle\cdot,\cdot\rangle:A_{0}\otimes A_{1}\rightarrow k$.
The multiplication tensor can be viewed then as the $\mathbb{Z}/3\mathbb{Z}$
- cyclically invariant linear function on $(\Pi A)^{\otimes3}$ 
\[
m_{A}:(\pi a_{1},\pi a_{2},\pi a_{3})\mapsto (-1)^{\bar{a} _{2}+1}\langle m_{2}(a_{1},a_{2}),a_{3}\rangle,~m_{A}\in(\ensuremath{\operatorname*{Hom}}(\left(\Pi A\right)^{\otimes3},k))^{\mathbb{Z}/3\mathbb{Z}}\text{.}
\]
The odd symmetric scalar product on $A$ corresponds to the odd anti-symmetric
product $\langle\cdot,\cdot\rangle^{\pi}$ on $\Pi A$: 
\[
\langle\pi a_{1},\pi a_{2}\rangle^{\pi}=(-1)^{\bar{a}_{1}+1}\langle a_{1},a_{2}\rangle
\]

The tensor product with the matrix algebra $gl_{N}$ is again naturally
a $\mathbb{Z}/2\mathbb{Z}-$ graded associative algebra with the odd
scalar product $A\otimes gl_{N}$. The cyclic tensor 
\begin{equation}
m_{A\otimes gl_{N}}\in(\ensuremath{\operatorname*{Hom}}(\left(\Pi A\otimes gl_{N}\right)^{\otimes3},k))^{\mathbb{Z}/3\mathbb{Z}}
\end{equation}
restricted to the diagonal $\Pi A\otimes gl_{N}\subset\left(\Pi A\otimes gl_{N}\right)^{\otimes3}$
is $GL(N)$-invariant cubic polynomial, denoted by $m_{A\otimes gl_{N}}\left(Z\right)$,
$Z\in\Pi A\otimes gl_{N}$. The associativity of the algebra $A$
translates into the equation 
\begin{equation}
\{m_{A\otimes gl_{N}}\left(Z\right),m_{A\otimes gl_{N}}\left(Z\right)\}=0,\label{mm}
\end{equation}
where $\{\cdot,\cdot\}$ is the odd Poisson bracket corresponding
to the odd anti-symmetric product $Tr\vert_{gl_{N}^{\otimes2}}\otimes\langle\cdot,\cdot\rangle^{\pi}$
on $\Pi A\otimes gl_{N}$.
\begin{proposition} 
The algebra of functions on $\Pi A\otimes gl_{N}$ is identified naturally,
preserving the odd Poisson bracket, with the algebra of polyvectors
on the even affine space $\Pi A_{1}\otimes gl_{N}$. $\square$  \end{proposition} 
This is analogous to algebra of functions on symplectic space being identified naturally, preserving Poisson bracket, with  algebra of functions on cotangent bundle of given lagrangian subspace.

Denote by $X^{\alpha}\in gl_{N}$,
$P_{\alpha}\in\Pi gl_{N}$ the matrices of coordinates on $\Pi A\otimes gl_{N}$
corresponding to a choice of a dual pair of bases $\{e^{\alpha}\}$,
$\{\xi_{\alpha}\}$ on $A_{0}$ and $A_{1}$ so that 
\begin{equation}
Z=\sum_{\alpha}\pi\xi_{\alpha}\otimes X^{\alpha}+\pi e^{\alpha}\otimes P_{\alpha}
\end{equation}
Then $(P_{\alpha})_{j}^{i}$ corresponds to the vector field $\frac{\partial}{\partial(X^{\alpha})_{i}^{j}}$
on $\Pi A_{1}\otimes gl_{N}$. The cubic polynomial $\frac{1}{3!} m_{A\otimes gl_{N}}\left(Z\right)$
corresponds to the sum of the function and the bivector, 
\begin{equation}
\frac{1}{3!}\sum_{\alpha,\beta,\gamma}(m_{A})_{\alpha\beta\gamma}Tr(X^{\alpha}X^{\beta}X^{\gamma})+\frac{1}{2}\sum_{\alpha,\beta,\gamma}(m_{A})_{\alpha}^{\beta\gamma}Tr(X^{\alpha}P_{\beta}P_{\gamma}).\label{sum}
\end{equation}

The odd Poisson bracket is generated by the odd second order Batalin-Vilkovisky
differential $\Delta$ acting on the algebra of functions on $\Pi A\otimes gl_{N}$
\begin{equation}
\{f_{1},f_{2}\}=(-1)^{\bar{f_{1}}}(\Delta(f_{1}f_{2})-\Delta(f_{1})f_{2}+(-1)^{\bar{f_{1}}}f_{1}\Delta(f_{2}))\label{bvident}
\end{equation}
\[
\Delta=\sum_{\alpha,i,j}\frac{\partial^{2}}{\partial X_{i}^{\alpha,j}\partial P_{\alpha,j}^{i}}
\]

\subsection{Divergence-free condition }

Let us assume from now on that the Lie algebra $A_{0}$ is unimodular.

\begin{condition} (unimodularity of $A_{0}$) For any $a\in A_{0}$
\begin{equation}
tr([a,\cdot]\vert_{A_{0}})=0\label{unim}
\end{equation}
\end{condition}

\begin{proposition} The unimodularity of $A_{0}$ (\ref{unim}) implies
\begin{equation}
\Delta m_{A\otimes gl_{N}}(Z)=0.\label{deltamZ}
\end{equation}
\end{proposition}$\square$ 

Next proposition is the standard corollary of the equations (\ref{mm}),(\ref{deltamZ})
and the relation (\ref{bvident}).

\begin{proposition} \label{deltatr} The exponent of the sum (\ref{sum})
is closed under the Batalin-Vilkovisky differential 
\[
\Delta\exp\left(\frac{1}{3!}m_{A\otimes gl_{N}}\left(Z\right)\right)=0.
\]
$\square$ \end{proposition}

\subsection{Closed De Rham differential form }

The affine space $\Pi A_{1}\otimes gl_{N}$ has a holomorphic volume
element, defined canonically up to a multiplication by a constant
\[
\varpi=\lambda\prod_{\alpha,i,j}dX_{i}^{\alpha,j}\text{.}
\]
It identifies the polyvectorfields on $\Pi A_{1}\otimes gl_{N}$ with the
de Rham differential forms $\Omega_{\Pi A_{1}\otimes gl_{N}}$ on
the same affine space via 
\[
\gamma\mapsto\gamma\vdash\varpi
\]
The Batalin-Vilkovisky differential $\Delta$ corresponds then to
the De Rham differential $d_{DR}$ acting on the differential forms.
By the proposition \ref{deltatr} the polyvector $\exp\frac{1}{3!}\left(m_{A\otimes gl_{N}}\left(Z\right)\right)$
defines the closed differential form 
\begin{equation}
\Psi(X)=\exp\left(\frac{1}{3!}\sum_{\alpha,\beta,\gamma}(m_{A})_{\alpha\beta\gamma}Tr(X^{\alpha}X^{\beta}X^{\gamma})+\frac{1}{2}\sum_{\alpha,\beta,\gamma}(m_{A})_{\alpha}^{\beta\gamma}Tr(X^{\alpha}\frac{\partial}{\partial X^{\beta}}\wedge\frac{\partial}{\partial X^{\gamma}})\right)\vdash\lambda\prod_{\alpha,i,j}dX_{i}^{\alpha,j}\label{psix}
\end{equation}
\[
d_{DR}\Psi(X)=0
\]
It is a sum of the closed differential forms of degrees $rN^{2}$
, $rN^{2}-2$ ,....

\subsection{Equivariantly closed differential form}

The unimodularity (\ref{unim}) implies the invariance of $\varpi$
under the co-adjoint action of the Lie algebra $A_{0}\otimes gl_{N}$
\[
X\mapsto [Y,X],
\]
$Y\in A_{0}\otimes gl_{N}$.
Consider the $A_{0}\otimes gl_{N}$ -equivariant differential forms
on $\Pi A_{1}\otimes gl_{N}$ : 
\[
\Omega_{\Pi A_{1}\otimes gl_{N}}^{A_{0}\otimes gl_{N}}=(\Omega_{\Pi A_{1}\otimes gl_{N}}\otimes\mathcal{O}_{A_{0}\otimes gl_{N}})^{A_{0}\otimes gl_{N}}\text{.}
\]
The $A_{0}\otimes gl_{N}$-equivariant differential is given by 
\[
d_{A_{0}\otimes gl_{N}}\Phi(Y)=d_{DR}\Phi-\sum_{\alpha,l,j}Y_{a,j}^{l}(i_{[E_{l}^{j}\otimes e^{a},\cdot]}\Phi)
\]
$\Phi\in\Omega_{\Pi A_{1}\otimes gl_{N}}^{A_{0}\otimes gl_{N}}$,
where $i_{\gamma}$ denotes the contraction operator with respect
to the vector field $\gamma$, see e.g. \cite{BV82}. This differential corresponds, when
passing to functions on $\Pi A\otimes gl_{N}$, to the sum 
\begin{eqnarray*}
\Delta_{A_{0}\otimes gl_{N}} & : & f(Z,Y)\mapsto \Delta f-\frac{1}{2}Tr\langle[Y,Z],Z\rangle^{\pi}f\text{},\text{}\\
f(Z,Y) & \in & (\mathcal{O}_{\Pi A\otimes gl_{N}}\otimes\mathcal{O}_{A_{0}\otimes gl_{N}})^{A_{0}\otimes gl_{N}}
\end{eqnarray*}
of the Batalin-Vilkovisky differential and the operator of multiplication
by the odd quadratic function 
\begin{equation}
\frac{1}{2}Tr\langle[Y,Z],Z\rangle^{\pi}=m_{A\otimes gl_{N}}(Y\otimes Z\otimes Z).
\end{equation}
The function depends on the equivariant parameters $Y\in A_{0}\otimes gl_{N}$.
\begin{theorem}
The product of the closed de Rham differential form $\Psi(X)$ (\ref{psix})
with the function $\exp Tr\langle Y,X\rangle$, $Y\in A_{0}\otimes gl_{N}$, $X\in\Pi A_{1}\otimes gl_{N}$, is $A_{0}\otimes gl_{N}$-equivariantly
closed differential form: 
\[
d_{A_{0}\otimes gl_{N}}(\exp(Tr\langle Y,X\rangle+\frac{1}{3!}m_{A\otimes gl_{N}}(X,\frac{\partial}{ \partial X}))\vdash\lambda\prod_{\alpha,i,j}dX_{i}^{\alpha,j})=0
\]
\end{theorem}
\begin{proof} Denote by $i_{m(X\frac{\partial}{\partial X}\frac{\partial}{\partial X})}$
the operator of contraction with the bivector field $\frac{1}{2}\sum_{\alpha,\beta,\gamma}(m_{A})_{\alpha}^{\beta\gamma}Tr(X^{\alpha}\frac{\partial}{\partial X^{\beta}}\wedge\frac{\partial}{\partial X^{\gamma}})$
and by $R_{Tr(YdX)}$ the operator of exterior multiplication by the
1-form $Tr\langle Y,dX\rangle$ acting on differential forms, 
\[
R_{Tr(YdX)}=[d_{DR},i_{Tr\langle Y,X\rangle}]
\]
where $i_{Tr\langle Y,X\rangle}$ is the multiplication by the linear
function $Tr\langle Y,X\rangle$. Then 
\[
[i_{m(X\frac{\partial}{\partial X}\frac{\partial}{\partial X})},R_{Tr(YdX)}]=i_{[\cdot,Y]}
\]
This is simply a particular case of the standard relation 
\[
[i_{\gamma_{1}},Lie_{\gamma_{2}}]=i_{[\gamma_{1},\gamma_{2}]}
\]
for the action of polyvector fields. Notice that 
\[
d_{DR}e^{Tr\langle Y,X\rangle}=e^{Tr\langle Y,X\rangle}\left(d_{DR}+R_{Tr\left(YdX\right)}\right)
\]
and that 
\begin{equation}
R_{Tr\left(YdX\right)}\exp\left(i_{m(X\frac{\partial}{\partial X}\frac{\partial}{\partial X})}\right)=\exp\left(i_{m(X\frac{\partial}{\partial X}\frac{\partial}{\partial X})}\right)\left(R_{Tr\left(YdX\right)}+i_{[\cdot,Y]}\right)
\end{equation}
Since $d_{DR}\Psi(X)=0$ , and $R_{Tr\left(YdX\right)}\prod_{\alpha,i,j}dX_{i}^{\alpha,j}=0$, therefore 
\[
d_{DR}(e^{Tr\langle Y,X\rangle}\Psi(X))=i_{[\cdot,Y]}e^{Tr\langle Y,X\rangle}\Psi(X). \,\,\,  \square
\]
\end{proof}

\section{The integral.}

The closed differential form $\Psi(X)$ is integrated over the cycles,
which are standard in the theory of exponential integrals $\int_{\Gamma}\exp f$, see (\cite{AVG} and references therein): 
\begin{equation}
\Gamma\in H_{\ast}(M,\ensuremath{\operatorname*{Re}}(f)\rightarrow-\infty),\,\, M=\Pi A_{1}\otimes gl_{N}(\mathbb{C}).
\end{equation}
Here $f$ is  the first term in (\ref{sum}), which
is the restriction of the cubic polynomial $\frac{1}{3!}m_{A\otimes gl_{N}}(Z)$
on $M$.

The relative homology are the same for such $f$, $f\neq 0$,
and for $f+Tr\langle Y,X\rangle$ since linear term is dominated by the cubic term when
$\vert X\vert\rightarrow+\infty$. Choosing a real form of $A_{0}\otimes gl_{N}\left(\mathbb{C}\right)$
and taking the cycles in $H_{\ast}(M,\ensuremath{\operatorname*{Re}}(f)\rightarrow-\infty)$
invariant with respect to this Lie algebra gives natural cycles for
integration of the equivariantly closed differential form $e^{Tr\langle Y,X\rangle}\Psi(X)$
\[
\mathcal{F}(Y)=\int_{\Gamma}\exp(Tr\langle Y,X\rangle+\frac{1}{3!}m_{A\otimes gl_{N}}(X,\frac{\partial}{\partial X}))\vdash\prod_{\alpha,i,j}dX_{i}^{\alpha,j}
\]
In general the integration cycles are the elements of the equivariant
homology 
\[
H_{\ast,A_{0}\otimes gl_{N}}(M,\ensuremath{\operatorname*{Re}}(\frac{1}{3!}\sum_{\alpha,\beta,\gamma}(m_{A})_{\alpha\beta\gamma}Tr(X^{\alpha}X^{\beta}X^{\gamma}))\rightarrow-\infty)
\]
One can consider also the normalized integral 
\begin{equation}
\widehat{\mathcal{F}}(Y)=\int_{\Gamma}\exp(Tr\langle Y,X\rangle+\frac{1}{3!}m_{A\otimes gl_{N}}(X,\frac{\partial}{\partial X}))\vdash\prod_{\alpha,i,j}dX_{i}^{\alpha,j}/\mathcal{F}_{\left[2\right]}(Y)\label{intgr}
\end{equation}
where $\mathcal{F}_{\left[2\right]}(Y)$ is the corresponding Gaussian
integral of the quadratic part of $f+Tr\langle Y,X\rangle$ at a critical
point $\left(-Y\right)^{\frac{1}{2}}$.

Let the associative algebra $A_{0}$ has an \emph{anti-involution}
$a\rightarrow a^{\dag}$ 
\[
\left(ab\right)^{\dag}=b^{\dag}a^{\dag}\text{,}\left(ca\right)^{\dag}=\overline{c}a^{\dag}\text{,}\ensuremath{\operatorname*{tr}}\left(a^{\dag}\right)=\overline{\ensuremath{\operatorname*{tr}}\left(a\right)}\text{,}\left(a^{\dag}\right)^{\dag}=a
\]
The anti-involution defines the natural  cycle for the equivariant integration.  This anti-involution extends naturally to $A_{0}\otimes gl_{N}\left(\mathbb{C}\right)$. Then the Lie subalgebra of anti-hermitian elements in $A_{0}\otimes gl_{N}\left(\mathbb{C}\right)$
\[
u_{N}(A_{0})=\left\{ Y^{\dag}=-Y\mid Y\in A_{0}\otimes gl_{N}\left(\mathbb{C}\right)\right\} 
\]
is a real form of $A_{0}\otimes gl_{N}\left(\mathbb{C}\right)$.
And the space of hermitian elements in the dual space 
\begin{equation}
\Gamma=\left\{ X^{\dag}=X\mid X\in A_{0}^{\vee}\otimes gl_{N}\left(\mathbb{C}\right)\right\} \label{Gam}
\end{equation}
is invariant under the action of $u_{N}(A_{0})$. Then the ``real-slice''
$\Gamma$ is the natural choice of the cycle for the equivariant integration.

The localization formula for equivariant cohomology reduces the integral
of the equivariantly closed form $\Omega$ over $\Gamma$ to the integral
over the fixed locus $F$, 
\begin{equation}
\int_{\Gamma}\Omega=\int_{F}\frac{\Omega}{eu\left(N_{F}\right)} \label{eNf}
\end{equation}
where $eu\left(N_{F}\right)$ is the euler class of the normal bundle
of $F$ in $\Gamma$, see \cite{AB84}, \cite{BV82}. Calculating
the integral using the equivariant localization leads to generalized
Vandermond determinants and $\tau-$functions.

Let for simplicity the algebra $A$ with odd scalar product is the
tensor product $A=A_{0}\otimes q_1$ of the even associative
algebra $A_{0}$ with scalar product, denoted $\eta\left(y_{1},y_{2}\right)$,
and the algebra $q_1=\left\{ 1,\xi\mid\xi^{2}=1\right\} $ with the
odd scalar product $\left\langle 1,\xi\right\rangle =1$.

Assume that the natural scalar product on the Lie algebra of anti-hermitian
elements in $A_{0}$ is positive definite 
\[
-\eta\left(y,y\right)=\eta\left(y,y^{\dag}\right)>0.
\]
Otherwise one can apply to $\Gamma$ a partial Wick rotation.

 Then the localization formula (\ref{eNf}) after some 
calculations leads to the following result (\cite{B18}): 
\begin{proposition}
The integral (\ref{intgr}), written in variables $t\in HC^{\ast}(A)$,
is a $\tau-$function of KP-type hierarchy and, in particular, satisfies the Hirota quadratic equations. 
\end{proposition}

\end{document}